\documentclass{article}
\usepackage{hyperref}
\usepackage{amsthm}
\usepackage{amssymb}
\usepackage{mathrsfs}
\usepackage{tikz}
\usetikzlibrary{cd}
\usepackage{comment}
\usepackage[utf8]{inputenc}
\usepackage[style=alphabetic]{biblatex}
\newcommand{\TITLE}{\(W\)-Types in Categories of Coalgebras}
\newcommand{\AUTHOR}{Taichi Uemura}

\newcommand{\cat}[1]{\mathcal{#1}}

\newcommand{\const}[1]{\mathrm{#1}}
\newcommand{\catconst}[1]{\mathbf{#1}}

\newcommand{\class}[1]{\mathscr{#1}}
\newcommand{\adj}{\dashv}
\newcommand{\Alg}{\const{Alg}}
\newcommand{\Poly}{\catconst{Poly}}
\newcommand{\To}{\Rightarrow}

\theoremstyle{plain}
\newtheorem{theorem}{Theorem}
\newtheorem{proposition}[theorem]{Proposition}
\newtheorem{lemma}[theorem]{Lemma}
\newtheorem{corollary}[theorem]{Corollary}
\theoremstyle{definition}
\newtheorem{definition}[theorem]{Definition}
\newtheorem{notation}[theorem]{Notation}
\theoremstyle{remark}
\newtheorem{example}[theorem]{Example}
\newtheorem{remark}[theorem]{Remark}

\title{\TITLE}
\author{\AUTHOR}

\hypersetup{
  pdftitle={\TITLE},
  pdfauthor={\AUTHOR}
}

\begin{document}

\maketitle

\begin{abstract}
  We construct \(W\)-types in the category of coalgebras for a
  cartesian comonad. It generalizes the constructions of \(W\)-types
  in presheaf toposes and gluing toposes.
\end{abstract}

\section{Introduction}
\label{sec:introduction}

\(W\)-type is an important feature of Martin-L\"{o}f's dependent type
theory \cite{martin-lof1982constructive,nordstrom1990programming} and
allows us to define a wide range of inductive types. Its categorical
counterpart is the initial algebra for a polynomial endofunctor
\cite{moerdijk-palmgren2000wellfounded,gambino2004wellfounded}.

In this paper we construct \(W\)-types in the category
of coalgebras for a cartesian comonad. This is a generalization of the
constructions of \(W\)-types in presheaf toposes and Freyd cover given
by Moerdijk and Palmgren \cite{moerdijk-palmgren2000wellfounded} and
gluing \(\Pi W\)-pretoposes given by van den Berg
\cite{vandenberg2006predicative}.

In Section \ref{sec:w-types-categories} we review the notion of
\(W\)-type. In Section \ref{sec:comonads} we describe pushforwards in
the category of coalgebras for a cartesian comonad. In Section
\ref{sec:w-types-coalgebras} we show the main theorem (Theorem
\ref{thm:w-types-of-coalgebras}) and give some
applications. In Section \ref{sec:algebr-equal-corefl} we show a
technical lemma used in the proof of the main theorem.

\section{\(W\)-Types in Categories}
\label{sec:w-types-categories}
\begingroup
\newcommand{\E}{\cat{E}}
\newcommand{\PP}{\class{P}}

We review the notions of polynomial, algebra and \(W\)-type. First we
recall the notion of algebra for an endofunctor.

\begin{definition}
  Let \(P : \E \to \E\) be an endofunctor on a category \(\E\). A
  \emph{\(P\)-algebra} consists of an object \(X \in \E\) and a
  morphism \(s_{X} : PX \to X\). A \emph{morphism} of \(P\)-algebras
  \((X, s_{X}) \to (Y, s_{Y})\) is a morphism \(h : X \to Y\) in
  \(\E\) such that \(h \circ s_{X} = s_{Y} \circ Ph\). We shall denote
  by \(\Alg(P)\) the category of \(P\)-algebras and morphisms of
  \(P\)-algebras, and by \(U_{P} : \Alg(P) \to \E\) the obvious
  forgetful functor. An \emph{initial \(P\)-algebra} is an initial
  object in the category \(\Alg(P)\).
\end{definition}

\begin{definition}
  Let \(P : \cat{E} \to \cat{E}\) and \(Q : \cat{F} \to \cat{F}\) be
  endofunctors. A \emph{lax morphism \(P \to Q\)} is a pair \((H,
  \sigma)\) consisting of a functor \(H : \cat{E} \to \cat{F}\) and a
  natural transformation \(\sigma : Q H \To H P\). A \emph{strong
    morphism \(P \to Q\)} is a lax morphism \((H, \sigma) : P \to Q\)
  such that \(\sigma\) is an natural isomorphism.
\end{definition}

\begin{proposition}
  Let \(P : \cat{E} \to \cat{E}\) and \(Q : \cat{F} \to \cat{F}\) be
  endofunctors. A lax morphism \((H, \sigma) : P \to Q\) induces a
  functor \(\tilde{H} : \Alg(P) \to \Alg(Q)\) such that
  \(U_{Q} \tilde{H} = H U_{P}\) defined by
  \(\tilde{H}(X, s) = (H X, H s \circ \sigma(X))\).
\end{proposition}

The following is analogous to the adjoint lifting theorem for
(co)algebras for (co)monads by Johnstone \cite{johnstone1975adjoint}
and Keigher \cite{keigher1975adjunctions}.

\begin{proposition}
  \label{prop:adjoint-lifting-for-endofunctor}
  Let \(P : \cat{E} \to \cat{E}\) and \(Q : \cat{F} \to \cat{F}\) be
  endofunctors and \((H, \sigma) : P \to Q\) a strong
  morphism. Suppose that \(H\) has a right adjoint \(R\). Then the
  induced functor \(\tilde{H} : \Alg(P) \to \Alg(Q)\) has a right
  adjoint \(\tilde{R} : \Alg(Q) \to \Alg(P)\) such that
  \(U_{P} \tilde{R} = R U_{Q}\).
\end{proposition}
\begin{proof}
  Let \(\eta\) and \(\varepsilon\) denote the unit and counit,
  respectively, of the adjunction \(H \adj R\). The functor \(R\) is
  part of a lax morphism \(Q \to P\) with the mate of \(\sigma^{-1}\)
  \[
    \begin{tikzcd}
      P R
      \arrow[r,"\eta P R"] &
      R H P R
      \arrow[r,"R \sigma^{-1} R"] &
      R Q H R
      \arrow[r,"R Q \varepsilon"] &
      R Q.
    \end{tikzcd}
  \]
  Hence we have a functor \(\tilde{R} : \Alg(Q) \to \Alg(P)\) such
  that \(U_{P} \tilde{R} = R U_{Q}\). One can show that \(\eta\) and
  \(\varepsilon\) defines natural transformations \(1 \To \tilde{R}
  \tilde{H}\) and \(\tilde{H} \tilde{R} \To 1\) respectively, so
  \(\tilde{R}\) is the right adjoint to \(\tilde{H}\).
\end{proof}

\(W\)-types are defined as the initial algebras for endofunctors
represented by polynomials.

\begin{notation}
  Let \(\E\) be a locally cartesian closed category. For a morphism
  \(f : I \to J\) in \(\E\), we shall denote by \(f^{*}\) the pullback
  functor \(\E/J \to \E/I\), by \(f_{!}\) or \(\Sigma^{\E}_{f}\) its
  left adjoint, and by \(f_{*}\) or \(\Pi^{\E}_{f}\) its right
  adjoint.
  \[
    \begin{tikzcd}
      \E/I \arrow[r,shift left=2ex,"f_{!}"]
      \arrow[r,shift right=2ex,"f_{*}"'] &
      \E/J \arrow[l,"f^{*}" description]
    \end{tikzcd}
    \quad
    f_{!} \adj f^{*} \adj f_{*}
  \]
  We call \(f_{*}\) the \emph{pushforward along \(f\)}.

  For a functor \(H : \cat{E} \to \cat{F}\) and an object \(I \in
  \cat{E}\), we denote by \(H/I : \cat{E}/I \to \cat{F}/H I\) the
  functor \((f : X \to I) \mapsto (H f : H X \to H I)\).
\end{notation}

\begin{definition}
  Let \(\E\) be a locally cartesian closed category. For objects \(I, J
  \in \E\), a \emph{polynomial} from \(I\) to \(J\) is a triple \((f,
  g, h)\) of morphisms of the form
  \[
    \begin{tikzcd}
      I &
      A \arrow[l,"h"'] \arrow[r,"g"] &
      B \arrow[r,"f"] &
      J.
    \end{tikzcd}
  \]
  For a polynomial \((f, g, h)\) from \(I\) to \(J\), we define a
  functor \(P_{f, g, h} : \E/I \to \E/J\) to be
  \(f_{!}g_{*}h^{*}\). We say a functor \(P : \E/I \to \E/J\) is
  represented by a polynomial \((f, g, h)\) from \(I\) to \(J\) when
  \(P\) is isomorphic to \(P_{f, g, h}\). We shall denote by
  \(\Poly_{\E}(I, J)\) the class of functors \(P : \E/I \to \E/J\)
  that are represented by some polynomial from \(I\) to \(J\). We
  often omit the subscript and simply write \(\Poly(I, J)\) when the
  category \(\E\) is clear from the context.
\end{definition}

\begin{definition}
  Let \(\E\) be a locally cartesian closed category. For an object
  \(I \in \E\), an \emph{endopolynomial} on \(I\) is a polynomial of
  the form
  \[
    \begin{tikzcd}
      I &
      A \arrow[l,"h"'] \arrow[r,"g"] &
      B \arrow[r,"f"] &
      I.
    \end{tikzcd}
  \]
  We say \(\E\) \emph{has \(W\)-types} or is a locally cartesian
  closed category \emph{with \(W\)-types} if, for any endopolynomial
  \((f, g, h)\) on any object \(I\), there exists an initial
  \(P_{f, g, h}\)-algebra.
\end{definition}

\begin{remark}
  \(W\)-types in this paper are what are called dependent \(W\)-types
  or indexed \(W\)-types in the literature. Gambino and Hyland
  \cite{gambino2004wellfounded} showed that a locally cartesian closed
  category has all non-dependent \(W\)-types if and only if it has all
  dependent \(W\)-types. So our terminology ``having \(W\)-types''
  coincides with that of Moerdijk and Palmgren
  \cite{moerdijk-palmgren2000wellfounded}.
\end{remark}

\begin{definition}
  Let \(\cat{E}\) and \(\cat{F}\) be locally cartesian closed categories
  and \(H : \cat{E} \to \cat{F}\) a locally cartesian closed functor,
  namely a functor that preserves finite limits and pushforwards. By
  definition, for an endopolynomial \((f, g, h)\) on an object
  \(I \in \cat{E}\), the functor \(H/I : \cat{E}/I \to \cat{F}/H I\)
  is part of a strong morphism \(P_{f, g, h} \to P_{H f, H g, H
    h}\). Suppose \(\cat{E}\) and \(\cat{F}\) have \(W\)-types. We say
  \(H\) \emph{preserves \(W\)-types} if the induced functor
  \(\Alg(P_{f, g, h}) \to \Alg(P_{H f, H g, H h})\) preserves initial
  objects for any endopolynomial \((f, g, h)\) in \(\cat{E}\).
\end{definition}

By Proposition \ref{prop:adjoint-lifting-for-endofunctor} we have the
following.

\begin{proposition}
  \label{prop:left-adjoint-preserves-w-types}
  Let \(\cat{E}\) and \(\cat{F}\) be locally cartesian closed
  categories with \(W\)-types and \(H : \cat{E} \to \cat{F}\) a
  locally cartesian closed functor. If \(H\) has a right adjoint, then
  \(H\) preserves \(W\)-types.
\end{proposition}

The class of functors represented by polynomials is closed under
several constructions.

\begin{proposition}
  \label{prop:polynomial-composition}
  Let \(\E\) be a locally cartesian closed category.
  \begin{enumerate}
  \item For any object \(I \in \E\), the identity functor \(\E/I \to
    \E/I\) belongs to \(\Poly(I, I)\).
  \item For functors \(P : \E/I \to \E/J\) and \(Q : \E/J \to \E/K\),
    if \(P\) belongs to \(\Poly(I, J)\) and \(Q\) belongs to
    \(\Poly(J, K)\), then \(QP\) belongs to \(\Poly(I, K)\).
  \end{enumerate}
\end{proposition}
\begin{proof}
  The identity functor \(\E/I \to \E/I\) is represented by the
  polynomial \((1_{I}, 1_{I}, 1_{I})\). For the composition, see
  \cite{gambino-kock2013polynomial}.
\end{proof}

\begin{proposition}
  \label{prop:polynomial-slice}
  Let \(\E\) be a locally cartesian closed category and
  \(P : \E/I \to \E/J\) a functor. If \(P\) belongs to
  \(\Poly(I, J)\), then \(P/K\) belongs to \(\Poly(K, PK)\) for any
  object \(K \in \E/I\).
\end{proposition}
\begin{proof}
  Suppose that \(P\) is represented by a polynomial
  \begin{tikzcd}
    I &
    A \arrow[l,"h"'] \arrow[r,"g"] &
    B \arrow[r,"f"] &
    J.
  \end{tikzcd}
  Consider the following diagram,
  \[
    \begin{tikzcd}
      & g^{*}g_{*}h^{*}K \arrow[r,"g'"]
      \arrow[d,"\varepsilon"']
      \arrow[ddr,phantom,"\lrcorner"{very near start,description}] &
      g_{*}h^{*}K \arrow[r,"\eta","\sim"']
      \arrow[dd] &
      PK \arrow[dd] \\
      K \arrow[d] &
      h^{*}K \arrow[d]
      \arrow[l,"h'"']
      \arrow[dl,phantom,"\llcorner"{very near start,description}] & & \\
      I &
      X \arrow[l,"h"] \arrow[r,"g"'] &
      Y \arrow[r,"f"'] &
      J
    \end{tikzcd}
  \]
  where the left and middle squares are pullbacks, \(\varepsilon\) is
  the counit of the adjunction \(g^{*} \adj g_{*}\), and
  \(\eta : g_{*}h^{*}K \to PK\) is an isomorphism. Then the functor
  \(P/K : \E/K \to \E/PK\) is represented by the polynomial
  \((\eta, g', h'\varepsilon)\).
\end{proof}

We introduce the notion of \emph{class of polynomial functors} which
is a generalization of the family of classes of functors
\((\Poly_{\E}(I, J))_{I, J \in \E}\). This generalization is necessary
to apply Theorem \ref{thm:w-types-of-coalgebras} on the existence of
\(W\)-types for a wide range of categorical constructions (see Section
\ref{sec:appl-intern-presh} and \ref{sec:application:-gluing}).

\begin{definition}
  Let \(\E\) be a locally cartesian closed category with
  \(W\)-types. A \emph{class of polynomial functors} consists of a class
  \(\PP(I, J)\) of functors \(\E/I \to \E/J\) for each pair of
  objects \(I, J \in \E\) satisfying the following conditions.
  \begin{enumerate}
  \item If \(P : \E/I \to \E/J\) belongs to \(\Poly(I, J)\), then it
    also belongs to \(\PP(I, J)\).
  \item If \(P : \E/I \to \E/J\) belongs to \(\PP(I, J)\) and \(Q :
    \E/J \to \E/K\) belongs to \(\PP(J, K)\), then \(QP\) belongs to
    \(\PP(I, K)\).
  \item If \(P : \E/I \to \E/J\) belongs to \(\PP(I, J)\) and \(P \cong
    Q\), then \(Q\) belongs to \(\PP(I, J)\).
  \item If \(P : \E/I \to \E/J\) belongs to \(\PP(I, J)\), then \(P/K :
    \E/K \to \E/PK\) belongs to \(\PP(K, PK)\) for any object \(K \in
    \E/I\).
  \item For any endofunctor \(P : \E/I \to \E/I\) that belongs to
    \(\PP(I, I)\), there exists an initial \(P\)-algebra.
  \end{enumerate}
\end{definition}

\begin{example}
  For a locally cartesian closed category \(\E\) with \(W\)-types,
  the classes \((\Poly(I, J))_{I, J \in \E}\) form a class of
  polynomial functors by Proposition \ref{prop:polynomial-composition}
  and \ref{prop:polynomial-slice}.
\end{example}

\begin{example}
  Let \(\PP\) be a class of polynomial functors on a locally cartesian
  closed category \(\E\) with \(W\)-types. For an object
  \(K \in \E\) and objects \(I, J \in \E/K\), we define a class
  \(\PP/K(I, J)\) of functors \((\E/K)/I \to (\E/K)/J\) to be the
  class of functors that belong to \(\PP(I, J)\) under the
  identification \((\E/K)/I \cong \E/I\) and \((\E/K)/J \cong
  \E/J\). Then \((\PP/K(I, J))_{I, J \in \E/K}\) form a class of
  polynomial functors on \(\E/K\).
\end{example}

\begin{remark}
  Clearly \(\Poly_{\E/K} \subset \Poly_{\E}/K\) holds, but
  \(\Poly_{\E}/K \subset \Poly_{\E/K}\) need not hold.
\end{remark}

\subsection{The Characterization Theorem}
\label{sec:char-theor}
\begingroup

In some nice category \(\E\), \(W\)-types have a nice characterization
due to van den Berg \cite{vandenberg2005inductive}. To achieve this
characterization, the category \(\E\) should be extensive
\cite{carboni-lack-walters1993extensive}.

\begin{definition}
  An \emph{extensive category} is a category \(\E\) with finite
  coproduct such that, for any objects \(A, B \in \E\), the coproduct
  functor
  \[
    + : \E/A \times \E/B \to \E/(A + B)
  \]
  is an equivalence of categories.
\end{definition}

When \(\E\) is locally cartesian closed, the extensivity of \(\E\) is
equivalent to much simpler conditions.

\begin{definition}
  Let \(\E\) be a category with finite coproducts. We say a binary
  coproduct \(A + B\) in \(\E\) is \emph{disjoint} if the inclusions
  \(A \to A + B\) and \(B \to A + B\) are monic and the pullback of
  these inclusions is the initial object.
\end{definition}

\begin{proposition}
  \label{prop:extensive-lccc}
  Let \(\E\) be a locally cartesian closed category with finite
  coproducts. The following conditions are equivalent.
  \begin{enumerate}
  \item \label{item:1}
    \(\E\) is extensive.
  \item \label{item:2}
    Binary coproducts in \(\E\) are disjoint.
  \item \label{item:3}
    The binary coproduct \(1 + 1\) is disjoint.
  \end{enumerate}
\end{proposition}
\begin{proof}
  It is known that a category with finite coproducts is extensive if
  and only if binary coproducts are universal and disjoint
  \cite[Proposition 2.14]{carboni-lack-walters1993extensive}. Binary
  coproducts in a locally cartesian closed category \(\E\) are always
  universal because pullback functors have right adjoints. Hence the
  conditions \ref{item:1} and \ref{item:2} are equivalent. Trivially
  \ref{item:2} implies \ref{item:3}. Suppose that the coproduct \(1 +
  1\) is disjoint. To show that any binary coproduct \(A + B\) is
  disjoint, it suffices to show that the pullback of the inclusions
  \(A \to A + B\) and \(B \to A + B\) is the initial object
  \cite[Proposition 2.13]{carboni-lack-walters1993extensive} because
  binary coproducts are universal. Now we have a morphism \(A
  \times_{A + B} B \to 1 \times_{1 + 1} 1 \cong 0\). Since \(\E\) is
  cartesian closed, the initial object is strict
  \cite[A1.5.12]{johnstone2002sketches}. Hence \(A \times_{A + B} B\)
  is the initial object.
\end{proof}

\begin{definition}
  Let \(\E\) be a category. We say an object \(A \in \E\) is
  \emph{well-founded} if any monomorphism \(X \rightarrowtail A\) into
  \(A\) is an isomorphism.
\end{definition}

\begin{definition}
  Let \(P : \cat{E} \to \cat{E}\) be an endofunctor on a category. A
  \emph{fixed point of \(P\)} is a \(P\)-algebra \((X, s)\) such that
  \(s : P X \to X\) is an isomorphism.
\end{definition}

\begin{theorem}[Characterization Theorem]
  \label{thm:w-types-characterization}
  Let \(\E\) be an extensive locally cartesian closed category with a
  natural number object and
  \begin{tikzcd}
    I &
    A \arrow[l,"h"']
    \arrow[r,"g"] &
    B \arrow[r,"f"] &
    I
  \end{tikzcd}
  an endopolynomial in \(\E\). For a \(P_{f, g, h}\)-algebra
  \((X, s_{X})\), the following conditions are equivalent.
  \begin{enumerate}
  \item \label{item:4}
    \((X, s_{X})\) is an initial \(P_{f, g, h}\)-algebra.
  \item \label{item:5}
    \((X, s_{X})\) is a well-founded fixed point of \(P_{f, g, h}\),
    namely a fixed point of \(P_{f, g, h}\) that is well-founded in
    the category \(\Alg(P_{f, g, h})\).
  \end{enumerate}
\end{theorem}
\begin{proof}
  The proof is similar to \cite[Theorem 26]{vandenberg2005inductive}.
\end{proof}
\endgroup
\endgroup

\section{Comonads}
\label{sec:comonads}
\begingroup

We review the theory of (cartesian) comonads.

\begin{notation}
  Let \(G\) be a comonad on a category \(\cat{E}\). We will refer to
  the counit and comultiplication of \(G\) as
  \(\varepsilon_{G} : G \To 1\) and \(\delta_{G} : G \To GG\)
  respectively. We denote by \(\cat{E}_{G}\) the category of
  \(G\)-coalgebras, by \(U_{G} : \cat{E}_{G} \to \cat{E}\) the
  forgetful functor and by \(F_{G}\) the right adjoint to \(U_{G}\)
  such that \(U_{G}F_{G} = G\). We will refer to the unit of the
  adjunction \(U_{G} \adj F_{G}\) as \(\eta_{G} : 1 \To F_{G}U_{G}\).
\end{notation}

\begin{definition}
  Let \(\cat{E}\) be a cartesian category. A \emph{cartesian
    comonad} on \(\cat{E}\) is a comonad \(G\) on \(\cat{E}\) whose
  underlying functor \(\cat{E} \to \cat{E}\) preserves finite limits.
\end{definition}

The following Lemma \ref{lem:coalgebra-create-colimits} is
well-known.

\begin{lemma}
  \label{lem:coalgebra-create-colimits}
  Let \(G\) be a comonad on a category \(\cat{E}\).
  \begin{enumerate}
  \item \label{item:6}
    The forgetful functor \(U_{G} : \cat{E}_{G} \to \cat{E}\) creates
    colimits.
  \item \label{item:7}
    If \(G\) is cartesian, then \(U_{G} : \cat{E}_{G} \to \cat{E}\)
    creates finite limits.
  \end{enumerate}
\end{lemma}

\begin{proposition}
  \label{prop:coalgebra-lccc}
  Let \(G\) be a cartesian comonad on a cartesian category
  \(\cat{E}\). If \(\cat{E}\) is locally cartesian closed, then so is
  \(\cat{E}_{G}\).
\end{proposition}

Proposition \ref{prop:coalgebra-lccc} is a well-known fact
\cite[A4.2.1]{johnstone2002sketches}. We concretely describe the
construction of pushforward functors.

First we recall the adjoint lifting theorem
\cite{johnstone1975adjoint}.

\begin{definition}
  Let \(\cat{E}\) and \(\cat{F}\) be categories, \(G\) a comonad on
  \(\cat{E}\) and \(H\) a comonad on \(\cat{F}\). An \emph{oplax
    morphism} \(G \to H\) consists of a functor \(S : \cat{E} \to
  \cat{F}\) and a natural transformation \(\sigma : SG \To HS\)
  satisfying \(\varepsilon_{H}S \circ \sigma = S\varepsilon_{G}\) and
  \(\delta_{H}S \circ \sigma = H\sigma \circ \sigma G \circ
  S\delta_{G}\).
\end{definition}

\begin{proposition}
  \label{prop:comonad-oplax-morphism-lift}
  Let \(\cat{E}\) and \(\cat{F}\) be categories, \(G\) a comonad on
  \(\cat{E}\), \(H\) a comonad on \(\cat{F}\) and
  \(S : \cat{E} \to \cat{F}\) a functor. A natural transformation
  \(\sigma : SG \To HS\) that makes \(S\) an oplax morphism
  \(G \to H\) bijectively corresponds to a functor
  \(T : \cat{E}_{G} \to \cat{F}_{H}\) such that \(U_{H}T = SU_{G}\).
\end{proposition}

\begin{proposition}[Adjoint Lifting Theorem]
  \label{prop:adjoint-lifting}
  Let \(\cat{E}\) and \(\cat{F}\) be cartesian categories, \(G\) a
  cartesian comonad on \(\cat{E}\), \(H\) a cartesian comonad on
  \(\cat{F}\) and \((S, \sigma) : G \to H\) an oplax morphism. Let
  \(T : \cat{E}_{G} \to \cat{F}_{H}\) denote the corresponding functor
  such that \(U_{H} T = S U_{G}\). If \(S : \cat{E} \to \cat{F}\) has
  a right adjoint \(R : \cat{F} \to \cat{E}\) with unit \(i\) and
  counit \(e\), then \(T\) has a right adjoint defined by the
  equalizer of the following coreflexive pair:
  \[
    \begin{tikzcd}
      F_{G} R U_{H}
      \arrow[r,"\varphi_{1}",shift left]
      \arrow[r,"\varphi_{2}"',shift right] &
      F_{G} R H U_{H}
      \arrow[r,"F_{G} R \varepsilon_{H} U_{H}"] &
      [6ex]
      F_{G} R U_{H}
    \end{tikzcd}
  \]
  \(\varphi_{1}\) is the composite
  \[
    \begin{tikzcd}
      F_{G} R U_{H}
      \arrow[r,"\eta_{G} F_{G} R U_{H}"] &
      [6ex]
      F_{G} U_{G} F_{G} R U_{H} = F_{G} G R U_{H}
      \arrow[r,"F_{G} \tilde{\sigma} U_{H}"] &
      [5ex]
      F_{G} R H U_{H}
    \end{tikzcd}
  \]
  where \(\tilde{\sigma} : G R \to R H\) is the mate of \(\sigma\),
  namely the composite
  \[
    \begin{tikzcd}
      G R \arrow[r,"i G R"] &
      R S G R \arrow[r,"R \sigma R"] &
      R H S R \arrow[r,"R H e"] &
      R H;
    \end{tikzcd}
  \]
  \(\varphi_{2}\) is \(F_{G} R U_{H} \eta_{H} : F_{G} R U_{H} \to
  F_{G} R U_{H} F_{H} U_{H} = F_{G} R H U_{H}\).
\end{proposition}

A slice of a category of coalgebras is again a category of coalgebras
in the following sense. Let \(G\) be a cartesian comonad on a
cartesian category \(\cat{E}\). For a \(G\)-coalgebra
\((A, \alpha)\), we define a functor
\(G_{\alpha} : \cat{E}/A \to \cat{E}/A\) to be the composite
\[
  \begin{tikzcd}
    \cat{E}/A \arrow[r,"G/A"] &
    \cat{E}/GA \arrow[r,"\alpha^{*}"] &
    \cat{E}/A.
  \end{tikzcd}
\]
One can show that \(G_{\alpha}\) is equipped with a comonad
structure. By definition \(G_{\alpha}\) is cartesian. For a morphism
of \(G\)-coalgebras \(f : (A, \alpha) \to (B, \beta)\), the canonical
natural transformation \(G_{\alpha}f^{*} \To f^{*}G_{\beta}\) is an
isomorphism. The category \((\cat{E}/A)_{G_{\alpha}}\) of
\(G_{\alpha}\)-coalgebras is naturally isomorphic to the slice
category \(\cat{E}_{G}/(A, \alpha)\).

Now we can describe pushforwards in a category of coalgebras. Let
\(f : (A, \alpha) \to (B, \beta)\) be a morphism of
\(G\)-coalgebras. We identify slice categories
\(\cat{E}_{G}/(A, \alpha)\) with \((\cat{E}/A)_{G_{\alpha}}\). By
Proposition \ref{prop:adjoint-lifting} the pushforward
\(\Pi^{\cat{E}_{G}}_{f}\) is the equalizer of the following
coreflexive pair:
\[
  \begin{tikzcd}
    F_{G_{\beta}} \Pi^{\cat{E}}_{f} U_{G_{\alpha}}
    \arrow[r,shift left,"\varphi_{1}"]
    \arrow[r,shift right,"\varphi_{2}"'] &
    F_{G_{\beta}} \Pi^{\cat{E}}_{f} G_{\alpha} U_{G_{\alpha}}
    \arrow[r,"F_{G_{\beta}} \Pi^{\cat{E}}_{f} \varepsilon_{G_{\alpha}}
    U_{G_{\alpha}}"] &
    [6ex]
    F_{G_{\beta}} \Pi^{\cat{E}}_{f} G_{\alpha}
  \end{tikzcd}
\]
\(\varphi_{1}\) is the composite
\[
  \begin{tikzcd}
    F_{G_{\beta}}\Pi^{\cat{E}}_{f}U_{G_{\alpha}}
    \arrow[r,"\eta_{G_{\beta}}
    F_{G_{\beta}}\Pi^{\cat{E}}_{f}U_{G_{\alpha}}"] &
    [6ex]
    F_{G_{\beta}}U_{G_{\beta}}F_{G_{\beta}}\Pi^{\cat{E}}_{f}U_{G_{\alpha}}
    = F_{G_{\beta}}G_{\beta}\Pi^{\cat{E}}_{f}U_{G_{\alpha}}
    \arrow[r,"F_{G_{\beta}}\sigma U_{G_{\alpha}}"] &
    F_{G_{\beta}}\Pi^{\cat{E}}_{f}G_{\alpha}U_{G_{\alpha}}
  \end{tikzcd}
\]
where \(\sigma : G_{\beta} \Pi^{\cat{E}}_{f} \to \Pi^{\cat{E}}_{f}
G_{\alpha}\) is the mate of the natural isomorphism \(f^{*} G_{\beta}
\cong G_{\alpha} f^{*}\); \(\varphi_{2}\) is \(F_{G_{\beta}}
\Pi^{\cat{E}}_{f} U_{G_{\alpha}} \eta_{G_{\alpha}} : F_{G_{\beta}}
\Pi^{\cat{E}}_{f} U_{G_{\alpha}} \to F_{G_{\beta}} \Pi^{\cat{E}}_{f}
U_{G_{\alpha}} F_{G_{\alpha}} U_{G_{\alpha}} = F_{G_{\beta}}
\Pi^{\cat{E}}_{f} G_{\alpha} U_{G_{\alpha}}\).

The following Proposition \ref{prop:algebra-of-coalgebra} describes
algebras in the category of coalgebras for a comonad.

\begin{proposition}
  \label{prop:algebra-of-coalgebra}
  Let \(G\) be a comonad on a category \(\cat{E}\) and \((P, \sigma)\)
  an oplax morphism \(G \to G\). Let
  \(Q : \cat{E}_{G} \to \cat{E}_{G}\) denote the corresponding functor
  such that \(U_{G}Q = PU_{G}\). The forgetful functor \(U_{G}\)
  induces a functor \(\Alg(Q) \to \Alg(P)\). Since \((G, \sigma)\) is
  a lax morphism \(P \to P\), we have a functor \(H : \Alg(P) \to
  \Alg(P)\) defined by \(H(X, s) = (GX, Gs \circ \sigma(X))\).
  \begin{enumerate}
  \item \label{item:12}
    \(H\) is part of a comonad on \(\Alg(P)\).
  \item \label{item:13}
    \(U_{P}H = GU_{P}\) and \((U_{P}, 1)\) is an oplax morphism \(H
    \to G\). Now we have a commutative diagram
    \[
      \begin{tikzcd}
        & \cat{E}_{G}
        \arrow[d,"U_{G}"] & \\
        \Alg(P)_{H}
        \arrow[ur]
        \arrow[dr,"U_{H}"'] &
        \cat{E} &
        \Alg(Q)
        \arrow[ul,"U_{Q}"']
        \arrow[dl] \\
        & \Alg(P)
        \arrow[u,"U_{P}"] &
      \end{tikzcd}
    \]
  \item \label{item:14}
    The category \(\Alg(Q)\) of \(Q\)-algebras is
    isomorphic over \(\cat{E}_{G} \times_{\cat{E}} \Alg(P)\) to the
    category \(\Alg(P)_{H}\) of \(H\)-coalgebras.
  \item \label{item:15}
    For any initial \(P\)-algebra \((W, s)\), there exists a unique
    initial \(Q\)-algebra \((W', s')\) such that \(U_{G}W' = W\) and
    \(U_{G}s' = s\).
  \end{enumerate}
\end{proposition}
\begin{proof}
  \ref{item:12} and \ref{item:13} are straightforward. To see
  \ref{item:14}, observe that both a \(Q\)-algebra and a
  \(H\)-coalgebra consist of the following data:
  \begin{itemize}
  \item an object \(A \in \cat{E}\);
  \item a morphism \(\alpha : A \to GA\);
  \item a morphism \(s : PA \to A\)
  \end{itemize}
  such that \(\alpha\) is a \(G\)-coalgebra and the following diagram
  commutes.
  \[
    \begin{tikzcd}
      PA \arrow[r,"P\alpha"]
      \arrow[d,"s"'] &
      PGA \arrow[r,"\sigma(A)"] &
      GPA \arrow[d,"Gs"] \\
      A \arrow[rr,"\alpha"'] & &
      GA
    \end{tikzcd}
  \]
  \ref{item:15} follows from \ref{item:14} and Lemma \ref{lem:coalgebra-create-colimits}.
\end{proof}

\begin{corollary}
  \label{cor:coalgebra-create-nno}
  Let \(G\) be a cartesian comonad on a cartesian category \(\cat{E}\)
  with finite coproducts. Then the forgetful functor
  \(U_{G} : \cat{E}_{G} \to \cat{E}\) creates natural number objects.
\end{corollary}
\begin{proof}
  By Lemma \ref{lem:coalgebra-create-colimits}, the diagram
  \[
    \begin{tikzcd}
      \cat{E}_{G} \arrow[r,"1_{\cat{E}_{G}} + (-)"]
      \arrow[d,"U_{G}"'] &
      \cat{E}_{G} \arrow[d,"U_{G}"] \\
      \cat{E} \arrow[r,"1_{\cat{E}} + (-)"'] &
      \cat{E}
    \end{tikzcd}
  \]
  commutes. Then use Proposition \ref{prop:algebra-of-coalgebra}.
\end{proof}
\endgroup

\section{\(W\)-Types in Categories of Coalgebras}
\label{sec:w-types-coalgebras}
\begingroup
\newcommand{\E}{\cat{E}}
\newcommand{\PP}{\class{P}}

We show the main theorem: the existence of \(W\)-types in the category
of coalgebras for a cartesian comonad. We also give some applications
of this theorem.

\begin{lemma}
  \label{lem:nno-as-w-type}
  Let \(\E\) be an extensive locally cartesian closed category. If
  \(\E\) has \(W\)-types, then it has a natural number object.
\end{lemma}
\begin{proof}
  Because the natural number object is the initial algebra for the
  associated functor of the polynomial
  \[
    \begin{tikzcd}
      1 &
      0 + 1 \arrow[l] \arrow[r] &
      1 + 1 \arrow[r] &
      1.
    \end{tikzcd}
  \]
\end{proof}

\begin{lemma}
  \label{lem:coalgebra-extensive}
  Let \(G\) be a cartesian comonad on a locally cartesian closed
  category \(\E\). If \(\E\) is an extensive category, then so is
  \(\E_{G}\).
\end{lemma}
\begin{proof}
  By Lemma \ref{lem:coalgebra-create-colimits}, the binary coproduct
  \(1 + 1\) in \(\cat{E}_{G}\) is disjoint. Then use Proposition
  \ref{prop:extensive-lccc}.
\end{proof}

\begin{lemma}
  \label{lem:algebra-for-equalizer}
  Let \(\E\) be a locally cartesian closed category,
  \(P_{0}, P_{1} : \E \to \E\) endofunctors and
  \(\varphi, \psi : P_{0} \To P_{1}\) and
  \(\varepsilon : P_{1} \To P_{0}\) natural transformations such that
  \(\varepsilon\varphi = \varepsilon\psi = 1\). Let
  \(\iota : P \To P_{0}\) denote the equalizer of \(\varphi\) and
  \(\psi\). Suppose that \(P_{0}\) preserves pullbacks and \(P_{1}\)
  preserves monomorphisms. If \(P_{0}\) and \(P_{1}\) have initial
  algebras, then \(P\) has a well-founded fixed point.
\end{lemma}

We will prove Lemma \ref{lem:algebra-for-equalizer} in Section
\ref{sec:algebr-equal-corefl}.

\begin{theorem}
  \label{thm:w-types-of-coalgebras}
  Let \(\E\) be an extensive locally cartesian closed category with
  \(W\)-types, \(\PP\) a class of polynomial functors on \(\E\) and
  \(G\) a cartesian comonad on \(\E\). Suppose that the underlying
  functor \(\E \to \E\) of \(G\) belongs to \(\PP(1, 1)\). Then the
  category \(\E_{G}\) of \(G\)-coalgebras has \(W\)-types.
\end{theorem}
\begin{proof}
  Let
  \[
    \begin{tikzcd}
      (I, \iota) &
      (A, \alpha) \arrow[l,"h"']
      \arrow[r,"g"] &
      (B, \beta) \arrow[r,"f"] &
      (I, \iota)
    \end{tikzcd}
  \]
  be an endopolynomial in \(\E_{G}\). Since \(f_{!}\) preserves
  equalizers, the endofunctor
  \(P_{f, g, h} : \E_{G}/(I, \iota) \to \E_{G}/(I, \iota)\) is an
  equalizer of the form
  \[
    \begin{tikzcd}
      P_{f, g, h} \arrow[r,tail] &
      P_{0} \arrow[r,shift left,"\varphi"]
      \arrow[r,shift right,"\psi"'] &
      P_{1}
    \end{tikzcd}
  \]
  where \(P_{0} = f_{!}F_{G_{\beta}}\Pi^{\E}_{f}U_{G_{\alpha}}h^{*}\)
  and
  \(P_{1} =
  f_{!}F_{G_{\beta}}\Pi^{\E}_{f}G_{\alpha}U_{G_{\alpha}}h^{*}\), and
  both \(\varphi\) and \(\psi\) are sections of a natural
  transformation \(P_{1} \To P_{0}\) (see Section
  \ref{sec:comonads}). Since the forgetful functor
  \(U_{G} : \E_{G} \to \E\) creates finite limits, we have
  \(P_{0} = f_{!}F_{G_{\beta}}\Pi^{\E}_{f}h^{*}U_{G_{\iota}}\) and
  \(P_{1} =
  f_{!}F_{G_{\beta}}\Pi^{\E}_{f}G_{\alpha}h^{*}U_{G_{\iota}}\). We
  define \(Q_{0} : \E/I \to \E/I\) to be
  \(f_{!}G_{\beta}\Pi^{\E}_{f}h^{*}\) and \(Q_{1} : \E/I \to \E/I\) to
  be \(f_{!}G_{\beta}\Pi^{\E}_{f}G_{\alpha}h^{*}\). Then we have
  \(U_{G_{\iota}}P_{0} =
  U_{G_{\iota}}f_{!}F_{G_{\beta}}\Pi^{\E}_{f}h^{*}U_{G_{\iota}} =
  f_{!}U_{\beta}F_{G_{\beta}}\Pi^{\E}_{f}h^{*}U_{G_{\iota}} =
  f_{!}G_{\beta}\Pi^{\E}_{f}h^{*}U_{G_{\iota}} = Q_{0}U_{G_{\iota}}\)
  and similarly \(U_{G_{\iota}}P_{1} = Q_{1}U_{G_{\iota}}\). Since
  \(Q_{0}\) and \(Q_{1}\) belong to \(\PP(I, I)\) by construction,
  they have initial algebras, and so do \(P_{0}\) and \(P_{1}\) by
  Proposition \ref{prop:algebra-of-coalgebra}. By construction
  \(P_{0}\) and \(P_{1}\) preserve pullbacks, and thus there exists a
  well-founded fixed point \((W, s)\) of \(P_{f, g, h}\) by Lemma
  \ref{lem:algebra-for-equalizer}. By Theorem
  \ref{thm:w-types-characterization} \((W, s)\) is an initial
  \(P_{f, g, h}\)-algebra, because \(\E_{G}\) has a natural number
  object by Lemma \ref{lem:nno-as-w-type} and Corollary
  \ref{cor:coalgebra-create-nno} and is extensive by Lemma
  \ref{lem:coalgebra-extensive}.
\end{proof}

\subsection{Application: Internal Diagrams}
\label{sec:appl-intern-presh}
\begingroup

Let \(C\) be an internal category in a locally cartesian closed
category \(\E\). We denote by \(\partial_{0}\) and \(\partial_{1}\)
the domain and codomain morphisms
\begin{tikzcd}
  C_{1} \arrow[r,shift left]
  \arrow[r,shift right] &
  C_{0}
\end{tikzcd}
respectively. The category \(\E^{C}\) of internal diagrams on \(C\) is
isomorphic to the category of coalgebras for the cartesian comonad
\[
  \begin{tikzcd}
    \E/C_{0} \arrow[r,"\partial_{1}^{*}"] &
    \E/C_{1} \arrow[r,"(\partial_{0})_{*}"] &
    \E/C_{0}
  \end{tikzcd}
\]
on \(\E/C_{0}\). It is known that the diagonal functor \(\Delta : \E
\to \E^{C}\) is a locally cartesian closed functor.

\begin{corollary}
  Let \(\E\) be an extensive locally cartesian closed category and
  \(C\) an internal category in \(\E\). If \(\E\) has \(W\)-types,
  then so does the category \(\E^{C}\) of internal diagrams on
  \(C\). Moreover, the diagonal functor \(\Delta : \E \to \E^{C}\)
  preserves \(W\)-types.
\end{corollary}
\begin{proof}
  Apply Theorem \ref{thm:w-types-of-coalgebras} with
  \(\PP = \Poly_{\E}/C_{0}\). The preservation of \(W\)-types follows
  from Proposition \ref{prop:left-adjoint-preserves-w-types} because
  \(\Delta\) has a right adjoint.
\end{proof}
\endgroup

\subsection{Application: Gluing}
\label{sec:application:-gluing}
\begingroup

Let \(H : \E_{1} \to \E_{2}\) be a cartesian functor between locally
cartesian closed categories. Since the comma category
\((\E_{2} \downarrow H)\) is isomorphic to the category of coalgebras
for the cartesian comonad
\((X_{1}, X_{2}) \mapsto (X_{1}, HX_{1} \times X_{2})\) on
\(\E_{1} \times \E_{2}\), it is locally cartesian closed. From the
construction of pushforwards described in Section \ref{sec:comonads}
we know that the projection \(\pi_{1} : (\E_{2} \downarrow H) \to
\E_{1}\) is a locally cartesian closed functor.

\begin{corollary}
  \label{cor:w-types-in-gluing}
  Let \(H : \E_{1} \to \E_{2}\) be a cartesian functor between
  extensive locally cartesian closed categories. If \(\E_{1}\) and
  \(\E_{2}\) have \(W\)-types, then so does \((\E_{2} \downarrow
  H)\). Moreover, the projection
  \(\pi_{1} : (\E_{2} \downarrow H) \to \E_{1}\) preserves
  \(W\)-types.
\end{corollary}

To show Corollary \ref{cor:w-types-in-gluing}, we define an appropriate
class of polynomial functors on \(\E_{1} \times \E_{2}\).

\begin{definition}
  Let \(\E_{1}\) and \(\E_{2}\) be locally cartesian closed categories
  with \(W\)-types and \(\PP_{1}\) and \(\PP_{2}\) classes of
  polynomial functors on \(\E_{1}\) and \(\E_{2}\) respectively. For
  objects \((I_{1}, I_{2}), (J_{1}, J_{2}) \in \E_{1} \times \E_{2}\),
  we define a class
  \(\PP_{1} \rtimes \PP_{2}((I_{1}, I_{2}), (J_{1}, J_{2}))\) of
  functors
  \((\E_{1} \times \E_{2})/(I_{1}, I_{2}) \cong \E_{1}/I_{1} \times
  \E_{2}/I_{2} \to \E_{1}/J_{1} \times \E_{2}/J_{2} \cong (\E_{1}
  \times \E_{2})/(J_{1}, J_{2})\) to be the class of functors of the
  form \((F_{1} \circ \pi_{1}, F_{2})\) with
  \(F_{1} : \E_{1}/I_{1} \to \E_{1}/J_{1}\) and
  \(F_{2} : \E_{1}/I_{1} \times \E_{2}/I_{2} \to \E_{2}/J_{2}\) such
  that \(F_{1}\) belongs to \(\PP_{1}(I_{1}, J_{1})\) and
  \(F_{2}(K, -)\) belongs to \(\PP_{2}(I_{2}, J_{2})\) for every
  \(K \in \E_{1}/I_{1}\).
\end{definition}

\begin{proposition}
  \label{prop:semiproduct-polynomials}
  For classes \(\PP_{1}\) and \(\PP_{2}\) of polynomial
  functors on locally cartesian closed categories with \(W\)-types
  \(\E_{1}\) and \(\E_{2}\) respectively, \(\PP_{1}
  \rtimes \PP_{2}\) is a class of polynomial functors on
  \(\E_{1} \times \E_{2}\).
\end{proposition}
\begin{proof}
  A functor of the form \(F_{1} \times F_{2}\) with
  \(F_{1} \in \PP_{1}(I_{1}, J_{1})\) and
  \(F_{2} \in \PP_{2}(I_{2}, J_{2})\) belongs to
  \(\PP_{1} \rtimes \PP_{2}((I_{1}, I_{2}), (J_{1}, J_{2}))\), and
  thus
  \(\Poly_{\E_{1} \times \E_{2}}((I_{1}, I_{2}), (J_{1}, J_{2}))
  \subset \PP_{1} \rtimes \PP_{2}\). It is also clear that
  \(\PP_{1} \rtimes \PP_{2}\) is closed under isomorphism. To see that
  it is closed under composition, let
  \(F_{1} : \E_{1}/I_{1} \to \E_{1}/J_{1}\),
  \(F_{2} : \E_{1}/I_{1} \times \E_{2}/I_{2} \to \E_{2}/J_{2}\),
  \(G_{1} : \E_{1}/J_{1} \to \E_{1}/K_{1}\) and
  \(G_{2} : \E_{1}/J_{1} \times \E_{2}/J_{2} \to \E_{2}/K_{2}\) be
  functors such that \(F_{1} \in \PP_{1}(I_{1}, J_{1}))\),
  \(G_{1} \in \PP_{1}(J_{1}, K_{1})\),
  \(F_{2}(X_{1}, -) \in \PP_{2}(I_{2}, J_{2})\) and
  \(G_{2}(Y_{1}, -) \in \PP_{2}(J_{2}, K_{2})\) for any
  \(X_{1} \in \E_{1}/I_{1}\) and \(Y_{1} \in \E_{1}/J_{1}\). Then we
  have
  \((G_{1}\pi_{1}, G_{2}) \circ (F_{1}\pi_{1}, F_{2}) =
  (G_{1}F_{1}\pi_{1}, G_{2}(F_{1}\pi_{1}, F_{2}))\) and
  \(G_{1}F_{1} \in \PP_{1}(I_{1}, K_{1})\) and
  \(G_{2}(F_{1}\pi_{1}, F_{2})(X_{1}, -) = G_{2}(F_{1}(X_{1}), -)
  \circ F_{2}(X_{1}, -) \in \PP_{2}(I_{2}, K_{2})\) for any
  \(X_{1} \in \E_{1}/I_{1}\).

  To show that \(\PP_{1} \rtimes \PP_{2}\) is closed under slice, let
  \(F_{1} : \E_{1}/I_{1} \to \E_{1}/J_{1}\) and
  \(F_{2} : \E_{1}/I_{1} \times \E_{2}/I_{2} \to \E_{2}/J_{2}\) be
  functors and \((K_{1}, K_{2}) \in \E_{1}/I_{1} \times \E_{2}/I_{2}\)
  objects such that \(F_{1} \in \PP_{1}(I_{1}, J_{1})\) and
  \(F_{2}(X_{1}, -) \in \PP_{2}(I_{2}, J_{2})\) for any
  \(X_{1} \in \E_{1}/I_{1}\). Then we have \((F_{1}\pi_{1},
  F_{2})/(K_{1}, K_{2}) = ((F_{1}/K_{1})\pi_{1}, F_{2}/(K_{1},
  K_{2}))\) and \(F_{1}/K_{1} \in \PP_{1}(K_{1}, FK_{1})\). We also
  have \((F_{2}/(K_{1}, K_{2}))(X_{1}, -) \in \PP_{2}(K_{2},
  F_{2}(K_{1}, K_{2}))\) for any \(X_{1} \in \E_{1}/K_{1}\) because it
  is the composite
  \[
    \begin{tikzcd}[column sep=10ex]
      \E_{1}/K_{1} \arrow[r,"{F_{2}(X_{1}, -)/K_{2}}"] &
      \E_{2}/F_{2}(X_{1}, K_{2}) \arrow[r,"{(X_{1})_{!}}"] &
      \E_{2}/F_{2}(K_{1}, K_{2}).
    \end{tikzcd}
  \]

  Finally we show the existence of initial algebras. Let
  \(F_{1} : \E_{1}/I_{1} \to \E_{1}/I_{1}\) and
  \(F_{2} : \E_{1}/I_{1} \times \E_{2}/I_{2} \to \E_{2}/I_{2}\) be
  functors such that \(F_{1} \in \PP_{1}(I_{1}, I_{1})\) and
  \(F_{2}(X_{1}, -) \in \PP_{2}(I_{2}, I_{2})\) for any
  \(X_{1} \in \E_{1}/I_{1}\). Let \((W_{1}, s_{1})\) be the initial
  \(F_{1}\)-algebra and \((W_{2}, s_{2})\) the initial
  \(F_{2}(W_{1}, -)\)-algebra. Then the object
  \((W_{1}, W_{2}) \in \E_{1}/I_{1} \times \E_{2}/I_{2}\) and the
  morphism
  \((s_{1}, s_{2}) : (F_{1}W_{1}, F_{2}(W_{1}, W_{2})) \to (W_{1},
  W_{2})\) form a \((F_{1}\pi_{1}, F_{2})\)-algebra. Let
  \((t_{1}, t_{2}) : (F_{1}X_{1}, F_{2}(X_{1}, X_{2})) \to (X_{1},
  X_{2})\) be another \((F_{1}\pi_{1}, F_{2})\)-algebra. A morphism
  \(((W_{1}, W_{2}), (s_{1}, s_{2})) \to ((X_{1}, X_{2}), (t_{1},
  t_{2}))\) of \((F_{1}\pi_{1}, F_{2})\)-algebras consists of a
  morphism \(h_{1} : W_{1} \to X_{1}\) and \(h_{2} : W_{2} \to X_{2}\)
  such that \(h_{1} \circ s_{1} = t_{1} \circ F_{1}h_{1}\) and
  \(h_{2} \circ s_{2} = t_{2} \circ F_{2}(h_{1}, h_{2})\). The second
  condition is equivalent to that \(h_{2}\) is a morphism of
  \(F_{2}(W_{1}, -)\)-algebras
  \((W_{2}, s_{2}) \to (X_{2}, t_{2} \circ F_{2}(h_{1},
  X_{2}))\). Hence, there exists a unique pair of morphisms
  \((h_{1}, h_{2})\) satisfying these conditions by the initiality of
  \((W_{1}, s_{1})\) and \((W_{2}, s_{2})\).
\end{proof}

\begin{proof}[Proof of Corollary \ref{cor:w-types-in-gluing}]
  Apply Theorem \ref{thm:w-types-of-coalgebras} with
  \(\PP = \Poly_{\E_{1}} \rtimes \Poly_{\E_{2}}\). The preservation of
  \(W\)-types follows from Proposition
  \ref{prop:left-adjoint-preserves-w-types} because \(\pi_{1}\) has a
  right adjoint.
\end{proof}
\endgroup
\endgroup

\section{Algebras for the Equalizer of a Coreflexive Pair}
\label{sec:algebr-equal-corefl}
\begingroup
\newcommand{\E}{\cat{E}}

In this section we prove Lemma \ref{lem:algebra-for-equalizer}: the
existence of a well-founded fixed point for an endofunctor defined as
the equalizer of a coreflexive pair.

\begin{lemma}
  \label{lem:equalizer-induces-mono}
  Suppose that we are given the following commutative diagram in a
  cartesian category
  \[
    \begin{tikzcd}
      X \arrow[r,tail] &
      Y \arrow[r,shift left,"u"]
      \arrow[r,shift right,"v"']
      \arrow[d,"p"] &
      Z \arrow[d,tail,"q"] \\
      A \arrow[r,tail] &
      B \arrow[r,shift left,"f"]
      \arrow[r,shift right,"g"'] &
      C
    \end{tikzcd}
  \]
  where \(A\) is the equalizer of \(f\) and \(g\), \(X\) is the
  equalizer of \(u\) and \(v\), and \(q\) is a monomorphism. Then the
  induced morphism \(X \to A\) is the pullback of \(p\) along
  \(A \to B\).
\end{lemma}
\begin{proof}
  For a morphism \(h : W \to Y\), the following conditions are
  equivalent.
  \begin{itemize}
  \item \(h\) factors through \(X\)
  \item \(uh = vh\)
  \item \(quh = qvh\)
  \item \(fph = gph\)
  \item \(ph\) factors through \(A\)
  \end{itemize}
\end{proof}

\begin{proof}
  [Proof of Lemma \ref{lem:algebra-for-equalizer}]
  Let \(s_{0} : P_{0}W_{0} \to W_{0}\) and
  \(s_{1} : P_{1}W_{1} \to W_{1}\) be initial \(P_{0}\)- and
  \(P_{1}\)-algebras respectively. By the initiality of \(W_{0}\) and
  \(W_{1}\), we have morphisms \(u, v : W_{0} \to W_{1}\) and
  \(e : W_{1} \to W_{0}\) such that
  \(u \circ s_{0} = s_{1} \circ \varphi \circ P_{0}u\),
  \(v \circ s_{0} = s_{1} \circ \psi \circ P_{0}v\),
  \(e \circ s_{1} = s_{0} \circ \varepsilon \circ P_{1}e\) and
  \(eu = ev = 1\). Let \(i : W \to W_{0}\) be the equalizer of \(u\)
  and \(v\). Consider the diagram
  \[
    \begin{tikzcd}
      PW \arrow[r,tail,"\iota(W)"]
      \arrow[d,tail,"Pi"'] &
      P_{0}W \arrow[r,shift left,"\varphi(W)"]
      \arrow[r,shift right,"\psi(W)"']
      \arrow[d,tail,"P_{0}i"] &
      P_{1}W \arrow[d,tail,"P_{1}i"] \\
      PW_{0} \arrow[r,tail,"\iota(W_{0})"]
      \arrow[d,shift left,"Pu"]
      \arrow[d,shift right,"Pv"'] &
      P_{0}W_{0} \arrow[r,shift left,"\varphi(W_{0})"]
      \arrow[r,shift right,"\psi(W_{0})"']
      \arrow[d,shift left,"P_{0}u"]
      \arrow[d,shift right,"P_{0}v"'] &
      P_{1}W_{0} \arrow[d,shift left,"P_{1}u"]
      \arrow[d,shift right,"P_{1}v"'] \\
      PW_{1} \arrow[r,tail,"\iota(W_{1})"'] &
      P_{0}W_{1} \arrow[r,shift left,"\varphi(W_{1})"]
      \arrow[r,shift right,"\psi(W_{1})"'] &
      P_{1}W_{1}
    \end{tikzcd}
  \]
  where all rows are equalizers. Since \(P_{1}\) preserves
  monomorphisms, \(PW\) is the pullback of \(PW_{0}\) along
  \(P_{0}i : P_{0}W \to P_{0}W_{0}\) by Lemma
  \ref{lem:equalizer-induces-mono} . \(P_{0}\) also preserves
  monomorphisms because it preserves pullbacks. Therefore \(PW\) is
  the intersection of \(PW_{0}\) and \(P_{0}W\) in the poset of
  subobjects of \(P_{0}W_{0}\). For a morphism
  \(f : Z \to P_{0}W_{0}\), the following conditions are equivalent.
  \begin{enumerate}
  \item \label{item:8}
    \(f\) factors through \(PW\).
  \item \label{item:9}
    \(\varphi(W_{0}) \circ f = \psi(W_{0}) \circ f\) and \(P_{0}u
    \circ f = P_{0}v \circ f\).
  \item \label{item:10}
    \(P_{1}u \circ \varphi(W_{0}) \circ f = P_{1}v \circ \psi(W_{0})
    \circ f\).
  \item \label{item:11}
    \(\varphi(W_{1}) \circ P_{0}u \circ f = \psi(W_{1}) \circ P_{0}v
    \circ f\).
  \end{enumerate}
  The conditions \ref{item:8} and \ref{item:9} are equivalent by
  definition. The conditions \ref{item:10} and \ref{item:11} are
  equivalent by the naturality of \(\varphi\) and \(\psi\). From
  \ref{item:9}, we have \ref{item:10} as
  \(P_{1}u \circ \varphi(W_{0}) \circ f = P_{1}u \circ \psi(W_{0})
  \circ f = \psi(W_{1}) \circ P_{0}u \circ f = \psi(W_{1}) \circ
  P_{0}v \circ f = P_{1}v \circ \psi(W_{0}) \circ f\). \ref{item:10}
  and \ref{item:11} imply \ref{item:9} because \(eu = ev = 1\) and
  \(\varepsilon\varphi = \varepsilon\psi = 1\). Hence, \(PW\) is also
  an equalizer of \(\varphi(W_{1}) \circ P_{0}u\) and
  \(\psi(W_{1}) \circ P_{0}v\). In the following diagram,
  \[
    \begin{tikzcd}[column sep=10ex]
      PW \arrow[r,tail]
      \arrow[d,dotted,"s"'] &
      P_{0}W_{0} \arrow[r,shift left,"\varphi(W_{1}) \circ P_{0}u"]
      \arrow[r,shift right,"\psi(W_{1}) \circ P_{0}v"']
      \arrow[d,"s_{0}"] &
      P_{1}W_{1} \arrow[d,"s_{1}"] \\
      W \arrow[r,tail,"i"'] &
      W_{0} \arrow[r,shift left,"u"]
      \arrow[r,shift right,"v"'] &
      W_{1}
    \end{tikzcd}
  \]
  all rows are equalizers and \(s_{0}\) and \(s_{1}\) are
  isomorphisms, and thus the morphism \(PW \to W_{0}\) factors as
  \(i \circ s\) with an isomorphism \(s : PW \to W\).

  We show that \((W, s)\) is a well-founded \(P\)-algebra. Let
  \((X, t : PX \to X)\) be a \(P\)-algebra and
  \(j : X \rightarrowtail W\) a monomorphism that is a morphism of
  \(P\)-algebras. Let \(X_{0}\) denote the pushforward
  \(i_{*}X\). Since \(i\) is a monomorphism, we have a pullback
  \[
    \begin{tikzcd}
      X \arrow[r,tail] \arrow[d,tail] &
      X_{0} \arrow[d,tail] \\
      W \arrow[r,tail,"i"'] &
      W_{0}.
    \end{tikzcd}
  \]
  Hence, it suffices to show that the monomorphism \(X_{0} \to W_{0}\)
  is an isomorphism. To prove this, it is enough to see that \(X_{0}\)
  carries a \(P_{0}\)-algebra structure over
  \(s_{0} : P_{0}W_{0} \to W_{0}\), that is, the composite
  \begin{tikzcd}[column sep=3ex]
    P_{0}X_{0} \arrow[r] &
    P_{0}W_{0} \arrow[r,"s_{0}"] &
    W_{0}
  \end{tikzcd}
  factors through \(X_{0} \rightarrowtail W_{0}\).
  \[
    \begin{tikzcd}
      P_{0}X_{0} \arrow[d]
      \arrow[r,dotted] &
      X_{0} \arrow[d,tail] \\
      P_{0}W_{0} \arrow[r,"s_{0}"'] &
      W_{0}
    \end{tikzcd}
  \]
  Since \(X_{0} = i_{*}X\) and \(i^{*}P_{0}W_{0} \cong PW\), it is
  equivalent to that the composite
  \begin{tikzcd}[column sep=3ex]
    i^{*}P_{0}X_{0} \arrow[r] &
    i^{*}P_{0}W_{0} \cong PW \arrow[r,"s"] &
    W
  \end{tikzcd}
  factors through \(X \rightarrowtail W\). Since \(X\) has a
  \(P\)-algebra structure over \(s : PW \to W\), it suffices to show
  that \(i^{*}P_{0}X_{0} \to PW\) factors through
  \(PX \rightarrowtail PW\). Now \(P\) is the equalizer of
  \(\varphi, \psi : P_{0} \To P_{1}\) and
  \(P_{1}X \to P_{1}W\) is a monomorphism, and thus the square
  \[
    \begin{tikzcd}
      PX \arrow[r,tail] \arrow[d,tail] &
      P_{0}X \arrow[d,tail] \\
      PW \arrow[r,tail,"\iota(W)"'] &
      P_{0}W
    \end{tikzcd}
  \]
  is a pullback by Lemma \ref{lem:equalizer-induces-mono}. Hence, to
  show that \(i^{*}P_{0}X_{0} \to PW\) factors through \(PX\), it
  suffices to show that the composite
  \begin{tikzcd}[column sep=4ex]
    i^{*}P_{0}X_{0} \arrow[r] &
    PW \arrow[r,"\iota(W)"] &
    P_{0}W
  \end{tikzcd}
  factors through \(P_{0}X\). Since \(P_{0}\) preserves pullbacks, we
  have a pullback
  \[
    \begin{tikzcd}
      P_{0}X \arrow[r,tail] \arrow[d,tail] &
      P_{0}X_{0} \arrow[d,tail] \\
      P_{0}W \arrow[r,tail,"P_{0}i"'] &
      P_{0}W_{0},
    \end{tikzcd}
  \]
  and thus it is enough to show that the composite
  \begin{tikzcd}[column sep=4ex]
    i^{*}P_{0}X_{0} \arrow[r] &
    PW \arrow[r,tail,"\iota(W)"] &
    P_{0}W \arrow[r,tail,"P_{0}i"] &
    P_{0}W_{0}
  \end{tikzcd}
  factors through \(P_{0}X_{0} \to P_{0}W_{0}\). But this is true
  because \(i^{*}P_{0}X_{0}\) is the pullback
  \[
    \begin{tikzcd}
      i^{*}P_{0}X_{0} \arrow[rr] \arrow[d] & &
      P_{0}X_{0} \arrow[d] \\
      PW \arrow[r,"Pi"'] &
      PW_{0} \arrow[r,"\iota(W_{0})"'] &
      P_{0}W_{0}
    \end{tikzcd}
  \]
  and \(\iota(W_{0}) \circ Pi = P_{0}i \circ \iota(W)\) by the
  naturality of \(\iota\).
\end{proof}
\endgroup

\printbibliography

\end{document}